\documentclass[12pt]{amsart}
\usepackage{amsfonts, amssymb, amsgen, amsbsy, amstext, amsopn, amsmath,
latexsym, indentfirst,color,longtable,verbatim}%pxfonts
\usepackage{verbatim} % package needed to activate the ``\renewenvironment{comment}'' command
\newcommand{\B}{\mathbb{B}}

\newcommand{\G}{\mathbb{G}}

\newcommand{\M}{\mathbb{M}}
\newcommand{\N}{\mathbb{N}}

\newcommand{\R}{\mathbb{R}}

\newcommand{\cC}{\mathcal{C}}

\newcommand{\cF}{\mathcal{F}}
\newcommand{\cG}{\mathcal{G}}
\newcommand{\cH}{\mathcal{H}}

\newcommand{\cL}{\mathcal{L}}

\newcommand{\cO}{\mathcal{O}}
\newcommand{\cP}{\mathcal{P}}

\newcommand{\cS}{\mathcal{S}}

\newcommand{\cR}{\mathcal{R}}

\newcommand{\cV}{\mathcal{V}}
\newcommand{\cW}{\mathcal{W}}

\newcommand{\ep}{\varepsilon}

\newcommand{\ph}{\varphi}

\newcommand{\sm}{\setminus}

\newcommand{\lb}{{\big\lbrace}}
\newcommand{\rb}{{\big\rbrace}}

\newcommand{\ls}{\mbox{\large $($}}
\newcommand{\rs}{\mbox{\large $)$}}
\newcommand{\Ls}{\mbox{\Large $($}}
\newcommand{\Rs}{\mbox{\Large $)$}}

\newcommand{\diams}{\mbox{\rm diam}\;\!}
\newcommand{\diam}{\mbox{\rm diam}}

\renewcommand{\div}{\mbox{\rm div\,}}

\newcommand{\Tan}{\mbox{\rm Tan}}
\newcommand{\Id}{\mbox{\rm Id}}

\newcommand{\bcup}{\bigcup}

\newcommand{\res}{\mbox{\LARGE{$\llcorner$}}}
\newcommand{\lan}{\langle}
\newcommand{\ran}{\rangle}

\newcommand{\lra}{\longrightarrow}
\newcommand{\der}{\partial}

\newcommand{\avint}{\hbox{\vrule height3.5pt depth-2.8pt
width4pt}\mkern-12mu\int\nolimits}

\newcommand{\Rd}{{\mathcal F}_HE}
\newcommand{\Per}{|\partial_HE|}

\newtheorem{The}{Theorem}[section]

\newtheorem{Def}{Definition}[section]
\newtheorem{Rem}{Remark}[section]

\newtheorem{Exa}[The]{Example}

\begin{document}

\title
[A new differentiation, shape of the unit ball and perimeter measure]
{{ A new differentiation, shape of the unit ball and \\ perimeter measure }}
\author{Valentino Magnani}
\address{Valentino Magnani, Dipartimento di Matematica, Universit\`a di Pisa \\
Largo Bruno Pontecorvo 5 \\ I-56127, Pisa}
\email{magnani@dm.unipi.it}
\date{\today}
\thanks{The author acknowledges the support of the European Project ERC AdG *GeMeThNES*, Grant Agreement ${\mathrm N}^{\circ}$ 246923.}
\date{\today}

\begin{abstract}
We present a new blow-up method that allows for establishing the first general formula to compute the perimeter measure with respect to the spherical Hausdorff measure in noncommutative nilpotent groups. This result leads us to an unexpected relationship between the area formula with respect to a distance and the profile of its corresponding unit ball.
\end{abstract}

\maketitle

\tableofcontents

%\newpage

%
%
%
%
%
%              INTRODUCTION 
% 
% 
\section{Introduction}

In the last decade, the study of sub-Riemannian Geometry, in short SR Geometry, has known a strong impulse in different areas, from PDE and Control Theory to Differential Geometry and Geometric Measure Theory. In particular, a number of Riemannian problems has a sub-Riemannian interpretation in a large framework and this may lead to either foundational questions or to new viewpoints.

The challenging project of developing Geometric Measure Theory on SR manifolds has shown that both of these aspects can happen. With this aim in mind, finding a theory of area in SR Geometry represents the starting point of a demanding program.

Historically, since the seminal works by Carath\'eodory \cite{Carath1914} and Hausdorff \cite{Hausd1918}, many theories grew to study $k$-dimensional Lebesgue area, smoothness conditions for area and coarea formulae, extensions to Finsler spaces, etc. We mention only a few relevant references \cite{AmbKir1}, \cite{AmbKir00},  \cite{AFP2000}, \cite{BPV96}, \cite{Busem47}, \cite{Cesari1956}, \cite{DeGiorgi}, \cite{DePHardt2012}, \cite{Federer44}, \cite{Federer55}, \cite{Federer59}, \cite{Federer69}, \cite{Federer78}, \cite{GoffZie70}, \cite{Haj93}, \cite{Kir94}, \cite{MSZ03}, \cite{Mattila1995}, \cite{Preiss1987}, \cite{Rado1948}, \cite{White99}, to give a very small glimpse of the much wider literature in both old and new research lines.

In the modern view, the Riemannian surface area can be computed by Euclidean tools, since smooth subsets have Lipschitz parametrizations, the Rademacher theorem holds and change of variables formulae perfectly fit with the standard density given by the Riemannian metric. 

The previous techniques fail completely in the SR case and this is due to two basic difficulties. First of all, we may not have Lipschitz parametrizations, even for smooth subsets of a sub-Riemannian manifold. This forces the use of abstract differentiation theorems for measures, although the second obstacle comes up exactly at this stage. In fact, the Besicovitch covering theorem, shortly B.C.T., fails to hold precisely for two important distances, such as the sub-Riemannian distance and the Cygan-Kor\'anyi distance in the Heisenberg group, \cite{KorRei95}, \cite{SawWhe92}. Although there exist some special distances such that the B.C.T. holds, \cite{LeDoRigot14}, a complete development of Geometric Measure Theory in the SR framework has to include all homogeneous distances, or at least the most important ones. To keep this generality, we do not have any general theorem to differentiate an arbitrary Radon measure.

We will show how to overcome these difficulties for hypersurfaces and for finite perimeter sets in special classes of nilpotent Lie groups.
Our ambient space is the {\em homogeneous} stratified group, corresponding to a stratified Lie group equipped with a fixed homogeneous distance, see Section~\ref{Sect:Notions}. Here the natural problem is to compute the perimeter measure by the Hausdorff measure constructed with the homogeneous distance of the group. 

Homogeneous stratified groups are Ahlfors regular and satisfy a Poincar\'e inequality, hence their metric perimeter measure coincides with the variational perimeter \cite{Mir}, and the general results of \cite{Amb01} give the following formula
\begin{equation}\label{intro:perimBeta}
\Per=\beta\; \cS_0^{Q-1}\res\Rd\,,
\end{equation}
where $\beta$ is measurable, $Q$ is the Hausdorff dimension of the group, $E\subset\G$ is an h-finite perimeter set, $\Rd$ is the reduced boundary and $\Per$ is the variational perimeter measure on groups, see Section~\ref{Sect:AreaPerim} for more details. 
The $(Q-1)$-dimensional spherical Hausdorff measure with no geometric constant $\cS_0^{Q-1}$ is introduced in Definition~\ref{def:sphericalH}.

A variational notion of perimeter measure on SR manifolds has been introduced in \cite{AmbMag28Ghe}, where the divergence operator is only defined by the volume measure. This implies that the notion of perimeter measure in stratified groups can be introduced by the standard divergence, see \eqref{eq:PerE}.

Finding a geometric expression for $\beta$ is obviously the crucial question. When the group is the Euclidean space, the classical De Giorgi's theory \cite{DeGiorgi} proves that $\beta\equiv\omega_{n-1}$, where $\omega_{n-1}$ is the volume of the unit ball in $\R^{n-1}$. 
Here we remark the crucial role of the classical area formula, joined with the rectifiability of the reduced boundary.
Extensions to the case of Finsler spaces have been also established, \cite{BPV96}.

When we consider a noncommutative stratified group, there is a drastic change of the problem, where the classical area formula does not apply.
In fact, according to the examples of \cite{KirSer}, the reduced boundary $\Rd$ in general may not be rectifiable in the sense of 3.2.14 of \cite{Federer69}, so all of the known methods fail. 
At present there are no results to find $\beta$ when the spherical Hausdorff measure is replaced by the Hausdorff measure. On the other hand, some integral representations for Borel measures with respect to the spherical Hausdorff measure can be written.
\begin{The}[\cite{Mag30}]\label{the:metricspherical}
Let $X$ be a diametrically regular metric space, let $\alpha>0$ and let $\mu$ be a Borel regular measure over $X$ 
such that there exists a countable open covering of $X$ whose elements have $\mu$ finite measure. 
If $B\subset A\subset X$ are Borel sets and $\cS_{\mu,\zeta_{b,\alpha}}$ covers $A$ finely,
then $\theta^\alpha(\mu,\cdot)$ is Borel on $A$. In addition, if $\cS^\alpha(A)<+\infty$
and $\mu\res A$ is absolutely continuous with respect to $\cS^\alpha\res A$, then we have
\begin{equation}\label{eq:spharea}
 \mu(B)=\int_B \theta^\alpha(\mu,x)\,d\cS^\alpha(x)\,.
\end{equation}
\end{The}
%
%
%For more information on this result, we address the reader to the original paper. 
The general hypotheses of Theorem~\ref{the:metricspherical} are obviously satisfied in all stratified groups, since B.C.T. is not required.
The {\em spherical Federer density} $\theta^\alpha(\mu,\cdot)$ has been recently introduced in \cite{Mag30}. The crucial aspect is the explicit formula
\begin{equation}\label{eq:FDens} 
\theta^\alpha(\mu,x)=\inf_{\ep>0}\; \sup\bigg\{\frac{\mu(\B)}{c_\alpha \diam(\B)^\alpha}: x\in\B\in\cF_b, \diams\B<\ep\bigg\}\,,
\end{equation}
where $\cF_b$ denotes the family of all closed balls in $X$ and $c_\alpha>0$ plays the role of the geometric constant to be fixed
in relation to the geometry of the metric ball.
Then finding $\beta$ in \eqref{intro:perimBeta} exactly corresponds to find a geometric expression for $\theta^{Q-1}(\Per,\cdot)$,
depending on the fixed homogeneous distance $d$. 

Let us point out that the spherical Federer density may differ from the $\alpha$-density in the sense of 2.10.19 of \cite{Federer69},
as shown for instance in \cite{Mag30}. The latter density has been recently related to a new measure theoretic area formula, proved by Franchi, Serapioni and Serra Cassano, \cite{FSSC8}.
In this formula, the spherical Hausdorff measure of \eqref{eq:spharea} is replaced by the so-called {\em centered Hausdorff measure}, 
introduced by Saint Raymond and Tricot in \cite{SRayTri88}, see also \cite{Edg2007}. 

To compute the Federer density for the perimeter measure, we consider suitable $\G$-regular sets, see Definition~\ref{def:GR}.
From \cite{FSSC5}, in two step stratified groups the reduced boundary $\Rd$ can be covered by a countable union of these $\G$-regular sets, up to $\cH^{Q-1}$-negligible sets, namely it is $\G$-rectifiable. We use this approach since it allows us to differentiate the perimeter measure in a way that can be 
suitable also for potential extensions to higher codimensional submanifolds.

The $\G$-rectifiability of $\Rd$ also holds in special classes of higher step groups, \cite{Marchi}, and for all $C^1$ smooth open sets of arbitrary stratified groups, \cite{Mag5}. For these reasons, in order to include all of these cases, we state our results for all h-finite perimeter sets whose reduced boundary $\Rd$ is $\G$-rectifiable.
\begin{The}[Area formula for the perimeter measure]\label{the:AreaPerim}
Let $\G$ be a stratified group and let $E\subset\G$ be an $h$-finite perimeter set. If $\Rd$ is $\G$-rectifiable, then we have  
\begin{equation}\label{eq:perimBeta}
\Per=\beta( d,\nu_E)\, \cS_0^{Q-1}\res\Rd\,.
\end{equation}
\end{The}
Clearly, the $\G$-rectifiability of $\Rd$ is necessary and it is clearly a crucial fact also in the classical context of Euclidean spaces. 
However, in a general stratified group it is not yet clear whether all reduced boundaries $\Rd$ are automatically $\G$-rectifiable.
This is still an important open question.

Rather surprisingly, formula \eqref{eq:perimBeta} holds for an arbitrary homogeneous distance, with no regularity assumption, and it also finds an interesting relationship with the shape of the metric unit ball generated by the distance.
As an example of this fact, Theorem~\ref{the:homConv} shows that whenever the unit ball $\B(0,1)$ is convex, then 
\begin{equation}\label{eq:dv}
 \beta(d,v)=\cH^{n-1}\ls N(v)\cap\B(0,1)\rs.
\end{equation}
Thus, it turns out that the Federer density and the $(Q-1)$-density coincide when the metric unit ball is convex.
This provides a simpler formula that relates perimeter measure and spherical Hausdorff measure. 
The key to prove \eqref{eq:dv} is a concavity-type property for the areas of all parallel one codimensional slices of a convex body,
see \cite{Busem49} and Theorem~\ref{the:convexSection}.

For homogeneous distances where $\beta(d,\cdot)$ is a constant function, it is natural to set $\omega_{\G,Q-1}=\beta(d,\cdot)$ and to define the spherical Hausdorff measure $\cS_\G^{Q-1}$ by this constant. Precisely, we define $c_{Q-1}=\omega_{\G,Q-1}/2^{Q-1}$ in Definition~\ref{def:sphericalH}.
We also introduce a class of homogeneous distances having this property. We call these distances {\em $V_1$-vertically symmetric}, since they
are equipped by a suitable group of symmetries modeled on $V_1$,  see Definition~\ref{def:hsym}. By Theorem~\ref{the:constBeta}, all of these distances have $\beta(d,\cdot)$ equal to a geometric constant and this gives a simpler representation of the perimeter measure.
\begin{The}[Area formula for symmetric distances]\label{the:AreaPerimG}
Let $\G$ be a stratified group equipped with a $V_1$-vertically symmetric distance and let $E\subset\G$ be set of $h$-finite perimeter. If $\Rd$ is
$\G$-rectifiable, then
\begin{equation}\label{eq:perimBetaG}
\Per=\cS_\G^{Q-1}\res\Rd\,.
\end{equation}
\end{The}
In each stratified group, we can find a homogeneous distance whose unit ball $\B(0,1)$ coincides with an Euclidean ball of sufficiently small radius, see Theorem~2 of \cite{HebSik90}. In particular, this distance satisfies the hypotheses of Theorem~\ref{the:AreaPerimG}. Many other examples of important distances with this symmetry property are available. We mention for instance the Cygan-Kor\'anyi distance, the distance $d_\infty$ of \cite{FSSC5} and the sub-Riemannian distance in the Heisenberg group, see Section~\ref{Sect:homball} and Section~\ref{Sect:Hsym}.

%For all of these distances, formula \eqref{eq:perimBetaG} applies. 

%... notice Theorem~\ref{the:AreaPerimG} also includes finite dimensional Banach spaces, hence including classical results by a new approach.
% .. to notice that formula \eqref{eq:perimBetaG} also applies when $\G$ is the Euclidean space, and here it provides a different proof with respect to the classical one, since it replaces the use of the Euclidean area formula by a differentiation theorem for measures.

%This naturally yields another question on the general conditions that make a SR distance $V_1$-vertically symmetric.

%
%
%
%
%
%%%%%%%%%%%%%%%%%%%%%%%%%%%%%%%%%%%%%%%%%%%%%%%%%%%%%%%%%%%%%%%%%%%%%%%%%%%%%%%%%%%%%%%%%%%%%%%%%%%%%%%%%%%%%%%%%%
%
%
%
%
%
%
%
\section{Notation, terminology and basic facts}\label{Sect:Notions}
%
%
%
%
%
%
%%%%%%%%%%%%%%%%%%%%%%%%%%%%%%%%%%%%%%%%%%%%%%%%%%%%%%%%%%%%%%%%%%%%%%%%%%%%%%%%%%%%%%%%%%%%%%%%%%%%%%%%%%%%%%%%%%

A {\em stratified group} can be seen as a graded linear space $\G=V_1\oplus\cdots\oplus V_\iota$ with graded Lie algebra $\cG=\cV_1\oplus\cdots\oplus\cV_\iota$, satisfying the properties $[\cV_1,\cV_j]=\cV_{j+1}$ for all integers $j\ge0$, where $\cV_{i}=\{0\}$ for all $i>\iota$ and $V_\iota\neq\{0\}$. This terminology is due to Folland, \cite{Fol75}.
Stratified groups equipped with a sub-Riemannian distance are also well known as Carnot groups, according to the terminology introduced by P.~Pansu. 

A {\em homogeneous distance} $d$ on $\G$ is a continuous and left invariant distance with $d(\delta_rx,\delta_ry)=r\,d(x,y)$ for all $x,y\in\G$ and $r>0$. We define the open and closed balls \[ B(y,r)=\lb z\in\G: d(z,y)<r\rb\quad\mbox{and}\quad \B(y,r)=\lb z\in\G: d(z,y)\le r \rb\,. \]
The corresponding homogeneous norm is denoted by $\|x\|=d(x,0)$ for all $x\in\G$.
When we wish to stress that the stratified group is equipped with a homogeneous distance we may use the terminology {\em homogeneous stratified group}, although often the homogeneous distance will be understood.
In fact, throughout a homogeneous distance will be fixed and $\Omega$ will denote an open subset of a stratified group $\G$.
Notice that any homogeneous distance is bi-Lipschitz equivalent to the SR distance.
%
%
%cmt
\begin{comment} This grading automatically defines projections $\pi_j:\cG\to\cV_j$ that also extend to the single tangent spaces $T_p\G$ and take values in the fibers 
\[ H^j_x\G=\{X(x): X\in\cV_j\}.\]
With a slight abuse of notation, we will use the same symbol $\pi_j$ for these projections.
This allows to define for every $C^1$ smooth $f$ the {\em horizontal differential} without referring to a scalar product
\[
 d_hf(x)(v)=df(x)\circ\pi_1(v)
\]
for every $v\in T_x\G$. An equivalent way is to define $d_Hf$ only on $H^1_x\G$.
\end{comment}
%
%
%

In our terminology, a {\em $C_h^1$ smooth function} $f:\Omega\to\R$ on an open set $\Omega$ of a stratified group $\G$ has the property that for all $x\in\Omega$ and $X\in\cV_1$ the {\em horizontal derivative}
\[
 Xf(x)=\lim_{t\to0}\frac{f\ls \Phi^X_t(x)\rs-f(x)}{t}
\]
exists and it is continuous in $\Omega$, where $\Phi^X$ denotes the flow of $X$. We denote by $\cC_h^1(\Omega)$ the linear space of all $\cC^1_h$ smooth functions on $\Omega$. 

We also define $d_hf(x):H_x\G\to\R$ as follows
\[ d_hf(x)(w)=\cW f(x), \]
where $w\in H^1_x\G$ and $\cW\in\cV_1$ is the unique left invariant vector field such that $\cW(x)=w$.
For every $j=1,\ldots,\iota$, we have also defined 
\[ H^j_x\G=\{X(x)\in T_x\G\mid  X\in\cV_j\}.\]
The class of $C^1_h$ smooth functions has an associated implicit function theorem.
% 
%
%                  IMPLICIT FUNCTION THEOREM
%
%
\begin{The}[Implicit function theorem]\label{the:implicit}
Let $x\in\Omega$, $X\in\cV_1$ and $f\in\cC_h^1(\Omega)$ with $Xf(x)\neq0$. Let $N_1\subset V_1$ be the kernel of $d_hf(x)$, let $N=N_1\oplus V_2\oplus\cdots\oplus V_\iota$ and let $H=\R v_X$, where $v_X=\exp X$. Then we have an open set $V\subset N$ with $0\in V$, a continuous function $\ph:V\to H$ and an open neighborhood $U\subset\G$ of $x$, such that
\begin{eqnarray}\label{impleq}
f^{-1}\ls f(x)\rs\cap U=\big\{xn \ph(n)\mid n\in V\big\}.
\end{eqnarray}
%Furthermore, there exists a constant $\kappa>0$ such that the Lipschitz-type estimate \begin{eqnarray}\label{intrlip} d\big(\ph(n),\ph(n')\big)\leq \kappa \;d\big(\ph(n')^{-1}n^{-1}n'\ph(n')\big) \end{eqnarray} holds. In particular, the mapping $\ph$ is $1/\iota$-H\"older continuous with respect to the metrics $d$ in $D^H_{\overline{h},r}$ and $\|\cdot\|$ in $D^N_{\overline{n},s}$.
\end{The}
This result is an immediate consequence of the Euclidean implicit function theorem, once the proper system of coordinates is fixed. It shows that regular level sets of $C^1_h$ smooth functions are locally graphs with respect to the group operation,  \cite{FSSC4}. 
%
% 
%cmt
\begin{comment}
Setting $F(n,t)=f(n\exp t X)$, then $\der_t F(n,t)=Xf(n\exp tX)>0$ is also continuous and the implicit function theorem gives us $\ph(n)$ such that \[f(n,\ph(n))=f(n_0,t_0)\] when $n$ varies around $n_0$.
\end{comment}
%
A more general implicit function theorem for mappings between two stratified groups $\G$ and $\M$ can be also obtained.
The new algebraic and topological difficulties of this case are partially overcome using the topological degree and assuming special algebraic factorizations of the source space. Level sets of these mappings define the general class of $(\G,\M)$-regular sets of $\G$, see \cite{Mag5} and \cite{Mag14} for more information. For the purposes of this work, the next definition refers to the case $\M=\R$. In this special case, these sets have first appeared in \cite{FSSC4} and called $\G$-regular hypersurfaces.
\begin{Def}\label{def:GR}\rm
We say that a subset $\Sigma\subset\G$ is a {\em parametrized $\G$-regular hypersurface} if there exists $f\in\cC^1_h(\Omega)$ such that $d_hf$ is everywhere nonvanishing and $\Sigma=f^{-1}(0)$. 
%A {\em $\G$-regular hypersurface} $S\subset\G$ has a countable open covering such that the intersection of $S$ with each element of the covering is a parametrized $\G$-regular hypersurface.
\end{Def}
%
%
%
%
%                                     GRADED BASIS
%
%
A {\em graded basis} $(e_1,\ldots,e_n)$ of $\G$ is defined by assuming that the families of vectors
\[
(e_{m_{j-1}+1},e_{m_{j-1}+2},\ldots,e_{m_j})
\]
are bases of the subspaces $V_j$ and $m_j=\sum_{i=1}^j\dim V_i$ for every $j=1,\ldots,\iota$, where $m_0=0$. We also set $m=m_1$. 

In the sequel, a graded basis is fixed and the corresponding Lebesgue measure $\cL^n$ is automatically defined on $\G$. Its left invariance makes it proportional to the {\em Haar measure} $\mu$ of $\G$. 
We fix an auxiliary scalar product on $\G$, whose restriction to $V_1$ defines the left invariant sub-Riemannian metric $g$ on each horizontal fiber $H_x\G$. In particular, this scalar product is chosen to make the previous graded basis orthonormal. We denote by $|\cdot|$ the associated norm. The choice of this scalar product for instance appears in Definition~\ref{def:Nnull}, where we use the Euclidean $n-1$ dimensional Hausdorff measure $\cH^{n-1}$.
%
%
%
\begin{comment} 
It should be verified that our approach only needs a scalar product on $V_1$ that defines a corresponding left invariant sub-Riemannian metric $g$ defined on each horizontal fiber $H_x\G$. In particular, this scalar product is chosen to make the basis $(e_1,\ldots,e_m)$ of $V_1$ orthonormal. The fixed scalar product on $V_2\oplus\cdots\oplus V_\iota$ should only need for the Euclidean $n-1$ dimensional Hausdorff measure $\cH^{n-1}$, that is used for $\beta(d,\cdot)$ etc.
A problem to better understand is whether we can limit ourselves to consider more general scalar products
that only make $V_1$ orthogonal to $V_2\oplus\cdots V_\iota$. We use the smaller class of scalar products that
make all layers orthogonal since we are applying the result of \cite{FSSC5} that understand this assumption.
My guess is that one could consider more general scalar products.
\end{comment}
%
%
%
\begin{comment}
My guess is that this additional scalar product on $V_2\oplus\cdots\oplus V_\iota$ could be avoided by considering the product $i_X\omega$, where $\omega$ is the fixed volume form on $\G$ and $X$ is a fixed horizontal unit left invariant vector field.
\end{comment}

For each parametrized $\G$-regular hypersurface $\Sigma\subset\G$ defined by $f:\Omega\to\R$, we may define the {\em intrinsic measure} of $\Sigma$ through the {\em perimeter measure}
\begin{equation}\label{eq:sigmaSigma}
\sigma_\Sigma(O)=\sup\left\{\int_{E_f}\div X\,d\mu\, \Big|\, X\in C_c^1(O,H\G),\, |X|\le1  \right\}
\end{equation}
for any open set $O\subset\Omega$, where $E_f=\{x\in \Omega: f(x)<0\}$.
The symbol $\div$ denotes the divergence operator with respect to the Haar measure $\mu$, that in our coordinate system gives the standard divergence operator.
%
%
%
%cmt
\begin{comment} We notice that the definition of $\sigma_\Sigma$ may depend a priori on the defining function $f$, so its computation according to Theorem~\ref{the:areaSigma}, makes the role of $f$ immaterial.
\end{comment}
%
%
%
%
The {\em horizontal normal} for a parametrized $\G$-regular hypersurface $\Sigma$ with defining function $f$ is defined for each $y\in\Sigma$ as follows
\begin{eqnarray}\label{inwrep}
 \nu_\Sigma(y)=\frac{\nabla_Hf(y)}{|\nabla_Hf(y)|}\,,
 \end{eqnarray}
where $(X_1,\ldots,X_m)$ is an orthonormal basis of $\cV_1$ and $\nabla_Hf=(X_1f,\ldots,X_mf)$.
For our purposes we do not claim an orientation, hence a sign, for the horizontal normal.
The following definition requires a little regularity of the unit ball.
%
%
%                 VERTICAL SUBGROUPS AND GEOMETRIC CONSTANTS
%
%
\begin{Def}[Vertical subgroups and geometric constants]\label{def:Nnull}\rm
For each $\nu\in V_1\sm\{0\}$, we define its corresponding {\em vertical subgroup}
$N(\nu)=\nu^\bot\oplus V_2\oplus\cdots\oplus V_\iota$, where $\nu^\bot$ is the subspace
of $V_1$ that is orthogonal to $\nu$. Then we define
\[
 \beta(d,\nu)=\max_{z\in\B(0,1)}\cH^{n-1}\ls\B(z,1)\cap N(\nu)\rs.
\]
\end{Def}
%
%
%cmt
\begin{comment} 
A homogeneous distance $d$ is {\em $(n-1)$-vertically regular} if for each $\nu\in V_1\sm\{0\}$ there exists an element $z(\nu)\in\B(0,1)$ such that
\begin{equation}
 \cH^{n-1}\Ls N(\nu)\cap\der\B\ls z(\nu),1\rs\Rs=0\quad\mbox{and}\quad \beta(d,\nu)=\cH^{n-1}\ls\B(z(\nu),1)\cap N(\nu)\rs.
\end{equation}
\end{comment}
%
%
%
%  The following formula for the
%  perimeter measure with respect to graded coordinates holds
%
% \begin{eqnarray}\label{exper}
% |\der E|_H(U)=\int_A\frac{\sqrt{\sum_{j=1}^mX_jf(\Phi(\xi))^2}}
% {X_1f(\Phi(\xi))}\;d\xi\,.
% \end{eqnarray}
%\section{Geometric measures}

%%%%%%%%%%%%%%%%%%%%%%%%%%%%%%%%%%%%%%%%%%%%%%%%%%%%%%%%%%%%%%%%%%%%%%%%%%%%%%%%%%%%%%%%%%%%%%%%
%
%
%
%
%        SECTION ON THE UPPER BLOW-UP OF THE PERIMETER MEASURE
%
%
%
%
%
\section{Upper blow-up of the perimeter measure}
%
%
%
%
%
%
%
%
%%%%%%%%%%%%%%%%%%%%%%%%%%%%%%%%%%%%%%%%%%%%%%%%%%%%%%%%%%%%%%%%%%%%%%%%%%%%%%%%%%%%%%%%%%%

In this section we prove the central result for this work, namely, a new blow-up theorem
for the perimeter measure. We recall here our basic notation.
%
%                CARATHEODORY'S CONSTRUCTION
%
%
\begin{Def}\label{def:sphericalH}\rm 
Let $\cF\subset\cP(\G)$ be a nonempty class of closed sets, let $\alpha>0$ and consider $c_\alpha>0$ as a suitable geometric constant. 
If $\delta>0$ and $E\subset X$, we define
\begin{equation*}
\phi_\delta(E)=\inf \bigg\lbrace\sum_{j=0}^\infty c_\alpha\,\diam(B_j)^\alpha: E\subset \bcup_{j\in\N} B_j ,\, \diam(B_j)\le\delta,\,
B_j\in\cF \bigg\rbrace\,.
\end{equation*}
If $\cF$ is the family of closed balls, the {\em $\alpha$-dimensional spherical Hausdorff measure} is
\[
\cS^\alpha(E)=\sup_{\delta>0}\phi_\delta(E).
\]
When $c_\alpha=2^{-\alpha}$, we use the symbol $\cS_0^\alpha$.
In the case $\cF$ is the family of all closed sets, $\alpha=k$ is a positive integer less than the linear dimension of $\G$,
$c_k=\omega_k$ and the Euclidean distance is fixed on $\G$, then the previous construction yields the Euclidean $k$-dimensional Hausdorff measure on $\G$,
that we denote by  $\cH^k$.
\end{Def}
% 
%
%
%
%
%           THEOREM ON THE UPPER BLOW-UP
%
%
%
\begin{The}[Upper blow-up]\label{the:persig}
Let $\Sigma$ be a parametrized $\G$-regular hypersurface and let $x\in\Sigma$. If $\sigma_\Sigma$ is its associated perimeter measure, then\begin{equation}\label{intro:thetabeta}
 \theta^{Q-1}(\sigma_\Sigma,x)=\beta\ls d,\nu_\Sigma(x)\rs\,.
\end{equation}
\end{The}
\begin{proof} We consider $f$ as the defining function of $\Sigma$ and select $X_1\in\cV$ such that $X_1(x)$ has unit length and it is orthogonal to the kernel of $d_Hf(x)$. Let $X_2,\ldots,X_m\in\cV_1$ be such that $(X_2(x),\ldots,X_m(x))$ is an orthonormal basis of this kernel. By Theorem~\ref{the:implicit}, we have an open neighborhood $V\subset N$ of the origin, with $N=\ker d_hf(x)\oplus V_2\oplus\cdots\oplus V_\iota$, a continuous function $\ph:V\to\R$, the vector $v_{X_1}=\exp X_1\in\G$ and an open neighborhood $U\subset\G$ of $x$, such that
\begin{eqnarray}\label{eq:nphn}
\Sigma\cap U=\lb xn (\ph(n)e_{X_1})\mid n\in V\rb\,.
\end{eqnarray}
Up to changing the sign of $f$ and possibly shrinking $U$, we assume that 
$X_1f>0$ bounded away from zero on $U$.
%
%
%cmt
\begin{comment} 
$X_1f\geq\alpha>0$ everywhere on $U$. 
\end{comment}
%
%
To define the intrinsic measure of $\Sigma$, we define the open set
\begin{eqnarray*}
E_f=\{y\in U\mid f(y)<0\}=\lb xn(te_{X_1})\in U|\, n\in V,\, t<\ph(n)\}.
\end{eqnarray*}
From \cite{FSSC4}, this set has finite perimeter and defining the graph mapping $\Phi:V\to\Sigma$ by $\Phi(n)=xn\ls\ph(n)e_{X_1}\rs$ for every $n\in V$, we also have the formula
\begin{eqnarray*}
\sigma_\Sigma\ls \B(y,t)\rs=|\der E|_H\ls \B(y,t)\rs= \int_{\Phi^{-1}(\B(y,t))}
\frac{\sqrt{\sum_{j=1}^mX_jf(\Phi(\eta))^2}}{X_1f(\Phi(\eta))}\;d\cH^{n-1}(\eta)
\end{eqnarray*}
for $t>0$ small and $y\in U$. We make the change of variables $n=\Lambda_t\eta$, where $\Lambda_t:N\to N$ and
\[ \Lambda_t\eta=\sum_{j=2}^m t\,\eta_j e_j+\sum_{i=2}^\iota \sum_{j=m_{i-1}+1}^{m_i} t^i\eta_j e_j\,. \]
We notice that $\delta_t|_N=\Lambda_t$ and the Jacobian of $\Lambda_t$ is $t^{Q-1}$, where $Q$ is the Hausdorff dimension of $\G$. By a change of variables, we get
\begin{equation}\label{eq:sigmat}
\sigma_\Sigma\ls \B(y,t)\rs=t^{Q-1} \int_{\Lambda_{1/t}\big(\Phi^{-1}(\B(y,t))\big)}
\frac{\sqrt{\sum_{j=1}^mX_jf(\Phi(\Lambda_t\eta))^2}}{X_1f(\Phi(\Lambda_t\eta))}\;d\cH^{n-1}(\eta)\,.
\end{equation}
The general definition of spherical Federer's density, \cite{Mag30}, in our setting gives
\begin{equation}\label{eq:sphF}
\theta^{Q-1}(\sigma_\Sigma,x)=\inf_{r>0}\sup_{\substack{y\in\B(x,t)\\ 0<t<r}}\frac{\sigma_\Sigma\ls\B(y,t)\rs}{t^{Q-1}}\,.
\end{equation}
In view of \eqref{eq:sigmat}, to find $\theta^{Q-1}(\sigma_\Sigma,x)$ we first observe that the sets
\begin{equation}\label{eq:Lambdat}
\Lambda_{1/t}\ls\Phi^{-1}(\B(y,t))\rs=\Big\{\eta\in \Lambda_{1/t}V\mid \ls\delta_{1/t}(y^{-1}x)\rs\,\eta\,\bigg(\frac{\ph\ls\delta_t\eta\rs}{t}e_{X_1}\bigg) \in\B(0,1)\Big\}
\end{equation}
are uniformly bounded uniformly with respect to $t$. 
Then for $t>0$ sufficiently small we have a compact set $K_0$ such that
\begin{equation}\label{eq:key_inclusion}
 \Lambda_{1/t}\ls\Phi^{-1}(\B(y,t))\rs\subset K_0\,.
\end{equation}
The first consequence of this inclusion is that $\theta^{Q-1}(\sigma_\Sigma,x)<+\infty$, hence there exist a sequence $\{t_k\}\subset(0,+\infty)$ converging to zero and a sequence of elements $y_k\in\B(x,t_k)$ such that
\[ 
\theta^{Q-1}(\sigma_\Sigma,x)=\lim_{k\to\infty} \int_{\Lambda_{1/t_k}\big(\Phi^{-1}(\B(y_k,t_k))\big)}
\frac{\sqrt{\sum_{j=1}^mX_jf(\Phi(\Lambda_{t_k}\eta))^2}}{X_1f(\Phi(\Lambda_{t_k}\eta))}\;d\cH^{n-1}(\eta)\,.
\]
Possibly extracting a subsequence, there exists $z\in\B(0,1)$ such that
\begin{equation}\label{eq:deltak}
\delta_{1/t_k}(y_k^{-1}x)\to z^{-1}\in\B(0,1).
\end{equation}
Setting $S_z=N\cap\B(z,1)$, we wish to show that for each $w\in N\sm S_z$ there holds
\begin{equation}\label{eq:A_k}
 \lim_{k\to\infty}{\bf 1}_{A_k}(w)=0\,,
\end{equation}
where $A_k=\Lambda_{1/t_k}\ls\Phi^{-1}(\B(y_k,t_k))\rs$. 
For this, we have to prove that
\begin{equation}\label{eq:ph0}
\lim_{t\to0^+} \frac{\ph(\delta_tw)}{t}=0.
\end{equation}
Since $\ph$ is only continuous, this makes the proof of this limit more delicate.
We define 
$$
A(w)=\Big\{t\in\R\mid t>0,\,  \ph(\delta_tw)\neq0\Big\}\,.
$$
If $\overline{A(w)}$ does not contain zero, the limit \eqref{eq:ph0} becomes obvious.
If $0\in \overline{A(w)}$, then we choose an arbitrary infinitesimal sequence $\{\tau_k\}\subset A(w)$.
Using the stratified mean value inequality in (1.41) of \cite{FS82}, we notice that
%
%
%cmt
\begin{comment}
We define the linear mapping $L:\G\to\R$ by $L=L\circ\pi_1$, where $\pi_1:\G\to V_1$ is the canonical projection and
\[
 X_jL(x)=L(\exp X_j)=X_jf(x)\quad \mbox{for all $x\in\G$ and $j=1,\ldots,m$}\,.
\]
The Folland-Stein's stratified mean valued inequality applied to $f(z)-L(z)$ gives 
\begin{eqnarray*}
|\ls f(xy)-L(xy)\rs-\ls f(x)-L(x)\rs|&=&|f(xy)-f(x)-L(y)| \\
&\le & C \|y\|\,\max_{z: \|z\|\le b \|y\|} |X_jf(xz)-X_jf(x)|
\end{eqnarray*}
where $y\in V$ and $\diam V$ is so small that $\B(x,b\,\diam V)\subset U$.
For $y=\delta_{\tau_k}w$, since $L(\delta_{\tau_k}w)=0$, for $k$ large the previous
inequalities and the fact that $f(x)=0$ give
\begin{eqnarray*}
|\ls f(x\delta_{\tau_k}w)-L(x\delta_{\tau_k}w)\rs-\ls f(x)-L(x)\rs|&=&|f(x\delta_{\tau_k}w)| \\
&\le & C\, \tau_k\,\|w\|\,\max_{z: \|z\|\le b \tau_k\,\|w\|} |X_jf(xz)-X_jf(x)| \\
&=& \tau_k\,\|w\|\,o(1)\,,
\end{eqnarray*}
hence $|f(x\delta_{\tau_k}w)|\tau_k^{-1}=o(1)$ as $\tau_k\to0$.
\end{comment}
%
%
%
\begin{equation*}
\lim_{k\to\infty}\frac{f(x\delta_{\tau_k}w)}{\tau_k}=\lim_{k\to\infty}\frac{f(x\delta_{\tau_k}w)-f(x)}{\tau_k}=0\,.
\end{equation*}
Since $\ph(\delta_{\tau_k}w)\neq0$, we can multiply and divide by $\ph(\delta_{\tau_k}w)$,
getting 
\begin{equation}
0=\lim_{k\to\infty}\frac{f(x\delta_{\tau_k}w)}{\tau_k}= \Big(-X_1f(x)\Big)\,\lim_{k\to\infty}\frac{\ph(\delta_{\tau_k}w)}{\tau_k}\,.
\end{equation}
%
%
%cmt
\begin{comment}
\begin{eqnarray*}
0=\lim_{k\to\infty}\frac{f(x\delta_{\tau_k}w)}{\tau_k}&=&\lim_{k\to\infty}\frac{f(x\delta_{\tau_k}w)-f(x\delta_{\tau_k}w\ph(\delta_{\tau_k}w)e_{X_1})}{\tau_k} \\
&=&\lim_{k\to\infty}\Big(\frac{f(x\delta_{\tau_k}w)-f(x\delta_{\tau_k}w\ph(\delta_{\tau_k}w)e_{X_1})}{\ph(\delta_{\tau_k}w)}\frac{\ph(\delta_{\tau_k}w)}{\tau_k}\Big) \\
&=& \Big(-X_1f(x)\Big)\,\lim_{k\to\infty}\frac{\ph(\delta_{\tau_k}w)}{\tau_k}\,,
\end{eqnarray*}
where the last equality follows by Lagrange theorem and the continuity of $X_1f$. 
\end{comment}
This proves \eqref{eq:ph0}, hence \eqref{eq:A_k} follows.
We consider the following integral as the sum
\[
 \int_{A_k}
\frac{\sqrt{\sum_{j=1}^mX_jf(\Phi(\Lambda_{t_k}\eta))^2}}{X_1f(\Phi(\Lambda_{t_k}\eta))}\;d\cH^{n-1}(\eta)= I_k+J_k\,,
\]
where, introducing the density function 
\[ \alpha(t,\eta)=\big(X_1f(\Phi(\Lambda_t\eta))\big)^{-1}\sqrt{\sum_{j=1}^mX_jf(\Phi(\Lambda_t\eta))^2},\]
we have set 
\[
 I_k=\int_{A_k\cap S_z} \alpha(t_k,\eta)\,d\cH^{n-1}(\eta)\quad
\mbox{and}\quad 
 J_k=\int_{A_k\sm S_z}\alpha(t_k,\eta)\,d\cH^{n-1}(\eta)\,.
\]
In principle, when $w\in N\cap\der\B(z,1)$ we do not have information on the limit of ${\bf 1}_{A_k\cap S_z}(w)$ as $k\to\infty$. Taking into account \eqref{eq:Lambdat}, this depend on the geometry of $x^{-1}\Sigma\cap\B(0,1)$.
However, in this step we wish to prove only one inequality. Then we consider the following inequality
\begin{equation}\label{eq:I_kle}
I_k\le \int_{S_z}\alpha(t_k,\eta)\,d\cH^{n-1}(\eta)\,.
\end{equation}
Due to \eqref{eq:key_inclusion}, we have 
\[
J_k\le\int_{K_0\sm S_z}{\bf 1}_{A_k}(\eta)\,\alpha(t_k,\eta)\,d\cH^{n-1}(\eta)\,,
\]
the integrand goes to zero as $k\to\infty$ due to \eqref{eq:A_k} and it is uniformly bounded, then we can apply Lebesgue's convergence theorem, proving that $J_k\to0$. Again Lebesgue's theorem gives
\[
\lim_{k\to\infty}\int_{S_z}\alpha(t_k,\eta)\,d\cH^{n-1}(\eta)=\cH^{n-1}(S_z)\,.
\]
In fact, $X_jf(0)=0$ for $j=2,\ldots,m$, hence $\alpha(t_k,\eta)\to1$ as $k\to\infty$.
This gives 
\[
\theta^{Q-1}(\sigma_\Sigma,x)=\lim_{k\to\infty}\int_{A_k}\frac{\sqrt{\sum_{j=1}^mX_jf(\Phi(\Lambda_{t_k}\eta))^2}}{X_1f(\Phi(\Lambda_{t_k}\eta))}\;d\cH^{n-1}(\eta)\le \cH^{n-1}(S_z)\le\cH^{n-1}(S_{z_0})\,,
\] 
where $z_0\in\B(0,1)$ is such that 
\begin{equation}\label{eq:z_0vertreg}
\cH^{n-1}(S_{z_0})=\beta(d,\nu_\Sigma(x))\,.
\end{equation}
To prove the opposite inequality, we select $y^0_t=x\delta_tz_0\in\B(x,t)$ and fix $\lambda>1$. We observe that
\[
\sup_{0<t<r}\frac{ \sigma_\Sigma\ls\B(y^0_t,\lambda t)\rs}{(\lambda t)^{Q-1}}\le\sup_{\substack{y\in\B(x,t)\\ 0<t<\lambda r}}\frac{\sigma_\Sigma\ls\B(y,t)\rs}{t^{Q-1}}
\]
for every $r>0$, therefore the definition of spherical Federer density \eqref{eq:sphF} yields
\[
\limsup_{t\to0^+}\frac{ \sigma_\Sigma\ls\B(y^0_t,\lambda t)\rs}{(\lambda t)^{Q-1}}\le\theta^{Q-1}(\sigma_\Sigma,x)\,.
\]
Taking into account \eqref{eq:Lambdat}, we set
\begin{equation}
A^0_t=\Lambda_{1/\lambda t}\ls\Phi^{-1}(\B(y^0_t,\lambda t))\rs=
\Big\{\eta\in \Lambda_{1/\lambda t}V\mid \eta\,\bigg(\frac{\ph\ls\delta_{\lambda t}\eta\rs}{\lambda t}e_{X_1}\bigg) \in\B\ls  \delta_{1/\lambda}z_0,1\rs\Big\}\,,
\end{equation}
that implies
\[
 \frac{ \sigma_\Sigma\ls\B(y^0_t,\lambda t)\rs}{(\lambda t)^{Q-1}}=\int_{A^0_t}\alpha(\lambda t,\eta)\,d\cH^{n-1}(\eta)=
\frac{1}{\lambda^{Q-1}}\int_{\delta_\lambda A^0_t}\alpha(\lambda t,\delta_{1/\lambda}\eta)\,d\cH^{n-1}(\eta)\,.
\]
We have
\begin{equation}
\delta_\lambda A^0_t=\Big\{\eta\in \Lambda_{1/t}V\mid
 \eta\,\Big(\frac{\ph\ls\delta_t\eta\rs}{t}e_{X_1}\Big) \in\B(z_0,\lambda) \Big\}
\end{equation}
and observe that $\alpha(t,\eta)$ is well defined and bounded on $[0,\bar\ep]\times\ls N\cap B(z_0,\lambda )\rs$, for $\bar\ep$ sufficiently small.
Taking into account that 
\[
\lim_{t\to0^+} {\bf 1}_{\delta_\lambda A^0_t}(w)=1
\]
for all $w\in N\cap B(z_0,\lambda)$, along with the inequality
\[
\lambda^{1-Q}\int_{N\cap B(z_0,\lambda )}{\bf 1}_{\delta_\lambda A^0_t}(w)\,\alpha(\lambda t,\delta_{1/\lambda}\eta)\,d\cH^{n-1}(\eta)\le
\frac{ \sigma_\Sigma\ls\B(y^0_t,\lambda t)\rs}{(\lambda t)^{Q-1}}\,,
\]
it follows that
\[
\lambda^{1-Q}\cH^{n-1}\ls N\cap B(z_0,\lambda )\rs\le\theta^{Q-1}(\sigma_\Sigma,x).
\]
%
%
%cmt
\begin{comment}
Since $z_0^{-1}N$ is a vertical hyperplane and $\tau_0:N\to z_0^{-1}N$, with $\tau_0(w)=z_0^{-1}w$ is an isometry, we have that $\cH^{Q-1}\ls z_0^{-1}N\cap\der\B(0,1)\rs=0$ is equivalent to $\cH^{Q-1}\ls N\cap\der\B(z_0,1)\rs=0$, that is equivalent to $\cH^{n-1}\ls N\cap\der\B(z_0,1)\rs=0$.
\end{comment}
As $\lambda\to1^+$ the opposite inequality follows, hence concluding the proof.
\end{proof}
\section{Area formulae for the perimeter measure}\label{Sect:AreaPerim}
This section is devoted to establish the general relationship between perimeter measure and spherical Hausdorff measure.
Joining Theorem~\ref{the:persig} with Theorem~\ref{the:metricspherical}, we obtain the following result.
%
%
%
%               AREA FORMULA FOR (G,R)-REGULAR SETS
%
%
\begin{The}[Area formula] \label{the:areaSigma}
Let $\Sigma$ be a parametrized $\G$-regular hypersurface. If $\sigma_\Sigma$ is its associated perimeter measure \eqref{eq:sigmaSigma}, then 
\[
 \sigma_\Sigma=\beta(d,\nu_\Sigma)\,\cS_0^{Q-1}\res\Sigma\,.
\]
\end{The}
This theorem also yields an area formula for the perimeter measure. To present this result we introduce a few more definitions. A subset $S\subset \G$ is {\em $\G$-rectifiable} if there exists a countable family $\{\Sigma_j\mid j\in\N\}$ of parametrized $\G$-regular hypersurfaces $\Sigma_j$ such that \[\cS_0^{Q-1}\ls S\sm\bcup\Sigma_j\rs=0.\] A measurable set $E\subset\G$ has $h$-finite perimeter if 
\begin{equation}\label{eq:PerE}
\sup\left\{\int_E\div X\,d\mu\, \Big|\, X\in C_c^1(\G,H\G),\, |X|\le1  \right\}<+\infty.
\end{equation}
This allows for defining a finite Radon measure $\Per$ on $\G$, see for instance \cite{AmbMag28Ghe}. 
The {\em reduced boundary} $\Rd$ is the set of points $x\in\G$ such that there exists $\nu_E(x)\in H_0\G$ with $|\nu_E(x)|=1$ and
\[
 \lim_{r\to0^+}\frac{1}{\Per\ls\B(x,r)\rs}\int_{\B(x,r)}\nu_E(y)\,d\Per(y)=\nu_E(x)\,,
\]
where $\nu_E$ is the generalized inward normal of $E$, see \cite{FSSC5} for more information.
%
%
%
%                 PROOF OF THE AREA FORMULA FOR THE PERIMETER MEASURE
%
%
%
\begin{proof}[Proof of Theorem~\ref{the:AreaPerim}]
Since the perimeter measure is asymptotically doubling, the perimeter measure $\Per$ is concentrated on the reduced boundary $\Rd$, see \cite{Amb01}. Moreover, the $\G$-rectifiability of $\Rd$ implies the existence of a countable family $\{\Sigma_j\mid j\in\N\}$ of parametrized $\G$-regular hypersurfaces $\Sigma_j$ such that 
\[\cS_0^{Q-1}\Big(\Rd\sm\bcup\Sigma_j\Big)=0.\]
On each $\Sigma_j$ we define the Radon measure
\[
\mu_j=\beta( d,\nu_{\Sigma_j})\, \cS_0^{Q-1}\res\Sigma_j\,,
\]
hence the area formula of Theorem~\ref{the:areaSigma} gives $\mu_j=\sigma_{\Sigma_j}$. The argument of the upper blow-up theorem clearly simplifies in the case of the following centered blow-up, giving
\begin{equation}\label{eq:Sigmaj}
 \lim_{r\to0^+}\frac{\sigma_{\Sigma_j}\ls \B(y,r)\rs}{r^{Q-1}}=\cH^{n-1}\Ls N\ls\nu_{\Sigma_j}(y)\rs\cap\B(0,1)\Rs
\end{equation}
for all $y\in\Sigma_j$. From Theorem~1.2 of \cite{AKLD09}, there exists $\cR_E\subset\Rd$ such that 
\[ 
 \Per\big(\Rd\sm\cR_E\big)=0
\]
and for every $x\in\cR_E$ the following vertical halfspace
\[
 Z\ls\nu_E(x)\rs=\lb v\in\G\mid v=v_1+v_0,\, v_1\in V_1,\,v_0\in V_2\oplus\cdots\oplus V_\iota,\, \lan v_1,\nu_E(x)\ran<0\rb
\]
belongs to $\Tan(E,x)$, see Definition~6.3 of \cite{AKLD09}.
Since $\Per$ is asymptotically doubling, \cite{Amb01}, the perimeter measure $\Per$ can differentiate $\mu_j$, 
hence we can define
\begin{equation}\label{eq:lim1}
\sigma_E(x)=\lim_{r\to0^+}\frac{\mu_j\ls\B(x,r)\rs}{\Per\ls \B(x,r)\rs}
\end{equation}
for each $x\in\cR_E^1\subset\cR_E\cap \Sigma_j$, where $\Per\big((\cR_E\cap\Sigma_j)\sm\cR_E^1\big)=0$.
If $x\in\cR_E^1\subset\cR_E\cap \Sigma_j$, then $Z\ls\nu_E(x)\rs\in\Tan(E,x)$ and there exists
an infinitesimal positive sequence $(r_k)$ of radii such that
\begin{equation}\label{eq:limZx}
 |\der_HE_{x,r_k}|\ls \B(0,1)\rs\lra|\der_HZ\ls\nu_E(x)\rs|\ls \B(0,1)\rs
\quad\mbox{as}\quad k\to\infty\,.
\end{equation}
%
%
%cmt
\begin{comment}
Notice that from \cite{FSSC5} the previous limit is stated for the open ball, but we obviously have $|\der_HE_{x,r}|\ls \der B(0,t)\rs=0$ for all $t>0$ using homogeneity and the fact that $|\der_HE_{x,r}|$ is a finite measure.
\end{comment}
%
%
%
The boundary of $Z\ls\nu_E(x)\rs$ is $N\ls\nu_E(x)\rs$, that is a parametrized $\G$-regular hypersurface and there holds
\begin{equation}\label{eq:Zx}
|Z\ls\nu_E(x)\rs|\ls \B(0,1)\rs=\cH^{n-1}\Ls N\ls\nu_E(x)\rs\cap \B(0,1)\Rs\,.
\end{equation}
%
%
%cmt
\begin{comment}
This formula can be easily obtained by considering the function $f_0:\G\to\R$ defined by
\[ 
f_0(y)=\lan\pi_1(y),\nu_E(x)\ran
\]
and selecting an orthonormal basis of $\cV_1$ such that $X_1(0)=\nu_E(x)$. From the formula of \cite{FSSC4} for the representation of the perimeter measure 
%
\begin{eqnarray*}
|\der E|_H\ls \B(0,1)\rs= \int_{\Phi_0^{-1}(\B(0,1))}
\frac{\sqrt{\sum_{j=1}^mX_jf_0(\Phi_0(n))^2}}{X_1f_0(\Phi_0(n))}\;d\cH^{n-1}(n)=\cH^{n-1}\ls N\cap\B(0,1)\rs\,,
\end{eqnarray*}
since $\Phi_0(n)=n$. This formula could be obviously obtained also directly from the definition of perimeter measure. 
\end{comment}
%
%
%
Joining \eqref{eq:Sigmaj}, \eqref{eq:limZx} and \eqref{eq:Zx}, we obtain that $\sigma_E(x)=1$, 
hence \eqref{eq:lim1} immediately leads us to the conclusion.
%
%
%cmt
\begin{comment}
Let $B$ be a $\Per$ measurable set of $\G$. Consider $B_1=\Rd\cap B$ and $B_2=B\sm\Rd$. By \cite{Amb01}, the condition $\Per(B_2)=0$ implies $\cS^{Q-1}_0(B_2)=0$. Since $\Rd$ is $\G$-rectifiable, we have 
\[
 B_1=N\cup \bcup B_1\cap \Sigma_j=N\cup \bcup_{j=1}^\infty \bigg(\ls B_1\cap \Sigma_j\rs\sm\bigcup_{l=1}^{j-1}\Sigma_l\bigg)
\]
with $\cS^{Q-1}_0(N)=0$ and $N$ is disjoint to the union of elements $B_1\cap\Sigma_j$, hence
\[
 \Per(B)=\Per(B_1)=\sum_{j=1}^\infty\Per\bigg(\ls B_1\cap \Sigma_j\rs\sm\bigcup_{l=1}^{j-1}\Sigma_l\bigg)
\]
and we also have
\[
 \Per\bigg(\ls B_1\cap \Sigma_j\rs\sm\bigcup_{l=1}^{j-1}\Sigma_l\bigg)=
\mu_j\bigg(\ls B_1\cap \Sigma_j\rs\sm\bigcup_{l=1}^{j-1}\Sigma_l\bigg)\,,
\]
in that \eqref{eq:lim1} holds for all $x\in \ls B_1\cap \Sigma_j\rs\sm\bigcup_{l=1}^{j-1}\Sigma_l$.
Moreover, for each $x\in\Rd\cap\Sigma_j$ we have 
\[
 \nu_{\Sigma_j}(x)=\nu_E(x)\,,
\]
then
\[
 \mu_j\bigg(\ls B_1\cap \Sigma_j\rs\sm\bigcup_{l=1}^{j-1}\Sigma_l\bigg)=\int_{(B_1\cap \Sigma_j)\sm\bigcup_{l=1}^{j-1}\Sigma_l}\beta\ls d,\nu_E(x)\rs\, d\cS_0^{Q-1}(x)\,.
\]
We have proved that
\begin{eqnarray*}
 \Per(B)&=&\sum_{j=1}^\infty\int_{(B_1\cap \Sigma_j)\sm\bigcup_{l=1}^{j-1}\Sigma_l}\beta\ls d,\nu_E(x)\rs\, d\cS_0^{Q-1}(x) \\
&=&\int_{\bcup_{j=1}^\infty\Sigma_j\cap B_1}\beta\ls d,\nu_E(x)\rs\, d\cS_0^{Q-1}(x) \\
&=&\int_{B_1}\beta\ls d,\nu_E(x)\rs\, d\cS_0^{Q-1}(x) \\
&=&\int_B\beta\ls d,\nu_E(x)\rs\, d\cS_0^{Q-1}(x).
\end{eqnarray*}
This clearly proves our claim.
\end{comment}
\end{proof}

%
%
%
%
%
%          SECTION ON THE SHAPE OF THE HOMOGENEOUS AND PERIMETER MEASURE
%
%
%
\section{Vertical sections of convex homogeneous balls}\label{Sect:homball}
In this section we study those homogeneous distances with convex unit ball.
The following classical result of convex geometry will play a key role.
\begin{The}[\cite{Busem49}]\label{the:convexSection}
Let $H$ be an $n$-dimensional Hilbert space with $n\ge2$ and let $C$ be a compact convex set with nonempty interior that contains the origin.
Let $v\in H\sm\{0\}$ and let $N$ denote the orthogonal space to $v$. Then the function 
\[
\psi(t)=\big[\cH^{n-1}\ls C\cap (tv+N)\rs\big]^{1/(n-1)}
\]
is concave on the interval $\{t\in\R: C\cap (tv+N)\neq\emptyset\}$.
\end{The}
Thus, we are in the position to establish the result of this section.
\begin{The}\label{the:homConv}
If $d$ is a homogeneous distance such that the corresponding unit ball $\B(0,1)$ 
is convex and $v\in V_1\sm\{0\}$, then we have
\begin{equation}
 \beta(d,v)=\cH^{n-1}\ls N(v)\cap\B(0,1)\rs\,.
\end{equation}
\end{The}
\begin{proof}
Let us set $N=N(v)$, where $N(v)$ is the vertical subgroup, see Definition~\ref{def:Nnull}. The Euclidean Jacobian of the translation $\tau_z:N\to zN$ is one for all $z\in\G$, hence
\[
\cH^{n-1}\ls N\cap \B(z,1)\rs =\cH^{n-1}\ls \B(0,1)\cap z^{-1}N\rs\,.
\]
To study the previous function with respect to $z$, we introduce
\[
 a(z)=\cH^{n-1}\ls \B(0,1)\cap zN\rs\,.
\]
Defining $H=\R v$, we have two canonical projections 
$\pi_1:\G\to H$ and $\pi_2:\G\to N$ such that $y=\pi_1(y)\pi_2(y)$ for all $y\in\G$, see Proposition~7.6 of \cite{Mag14}.
Since $H$ is a horizontal subspace, one can also check that $\pi_1:\G\to H$ is precisely the linear projection onto $H$ with respect to the direct sum of linear spaces $H\oplus N=\G$. As a consequence,
\begin{equation}\label{eq:az}
 a(z)=\cH^{n-1}\ls \B(0,1)\cap zN\rs=\cH^{n-1}\ls \B(0,1)\cap\pi_1(z)N\rs\,.
\end{equation}
Furthermore, $\pi_1(z)N=\pi_1(z)+N$ and $\B(0,1)^{-1}=-\B(0,1)$, therefore
\[
  a(z)=\cH^{n-1}\ls \B(0,1)\cap \ls\pi_1(z)+N\rs\rs=\cH^{n-1}\ls \B(0,1)\cap\ls-\pi_1(z)+N\rs\rs
\]
and the property $-\pi_1(z)=\pi_1(z^{-1})$ yields
\[
  a(z)=\cH^{n-1}\ls \B(0,1)\cap \ls\pi_1(z^{-1})+N\rs\rs=\cH^{n-1}\ls \B(0,1)\cap\ls \pi_1(z^{-1})N\rs\rs.
\]
Thus, by \eqref{eq:az} we get 
$a(z)=\cH^{n-1}\ls \B(0,1)\cap\ls z^{-1}N\rs\rs=a(z^{-1})=a(-z)$, hence $a$ is an even function.
For every $t\in\R$ we may define the function 
\[
b(t)=\Big[\cH^{n-1}\Big( \B(0,1)\cap\ls tv+N\rs\Big)\Big]^{1/(n-1)}
\]
By Theorem~\ref{the:convexSection}, the function $b(t)=\sqrt[n-1]{a(tv)}$ is concave and even on the compact interval 
\[
 I=\{t\in\R: \B(0,1)\cap (tv+N)\neq\emptyset\}\,,
\]
hence we get
\[
 \beta(d,v)=\max_{z\in\B(0,1)}\cH^{n-1}\ls(\B(z,1)\cap N(\nu)\rs=\cH^{n-1}\ls N(v)\cap\B(0,1)\rs\,.
\]
Being $N(v)\cap\der \B(0,1)$ locally parametrized by Lipschitz mappings on an $(n-2)$ dimensional open set,
then we obviously have $\cH^{n-1}\ls N(v)\cap\der \B(0,1)\rs=0$, concluding the proof.
%
%
%
%cmt
\begin{comment}
Notice that we have used the fact that
\[
 N(v)\cap\der \B(0,1)=\der\ls N(v)\cap\B(0,1)\rs
\]
and this is true, since the interior of $N(v)\cap\B(0,1)$ is nonempty.
The subset $\der\ls N(v)\cap\B(0,1)\rs$ is the boundary of a convex set in $N$ having nonempty interior,
hence it is locally parametrized by a Lipschitz mapping.
\end{comment}
\end{proof}
In any general homogeneous group we can always find a homogeneous distance with convex unit ball,
\cite{HebSik90}. The next examples provide other distances with this property.
\begin{Exa}\label{exa:Koranyi}\rm
Let $N=V_1\oplus V_2$ be an H-type group. We have an explicit formula for a homogeneous distance $d(x,y)=\|x^{-1}y\|$ such that 
\[  \|x\|=\sqrt[4]{|x_1|^4+16|x_2|^2}\, \]
where $x,y\in N$, $x=x_1+x_2$ and $x_i\in V_i$ for $i=1,2$, see \cite{Cygan81}. The unit ball with respect to this distance is clearly a convex set.
\end{Exa}
\begin{Exa}\label{exa:dinfty}\rm
Let $\G=V_1\oplus V_2\oplus\cdots\oplus V_\iota$ be any stratified group. From the Baker-Campbell-Hausdorff formula
it is easy to see the existence of constants $\ep_j>0$, with $j=1,\ldots,\iota$ and $\ep_1=1$, such that setting 
\[  \|x\|=\max\{\ep_j |x_j|^{1/j}\} \]
with $x_i\in V_i$ for all $i=1,\ldots,\iota$, we have actually defined the homogeneous distance
$
d_\infty(x,y)=\|x^{-1}y\|, 
$
as it was observed in \cite{FSSC5}. The unit ball $\B(0,1)$ with respect to $d_\infty$ is clearly a convex set.
\end{Exa}
%
%
%cmt
\begin{comment} 
In fact, $d_\infty(x,0)=\max\{\ep_i|x_i|^{1/i}\}$ and 
\[ 
\{d_\infty(x,0)\le1\}=\{x:\,\ep_i|x_i|^{1/i}\le1\,\mbox{for all $i$} \}=
\{x:\,\ep_i^i|x_i|\le1\,\mbox{for all $i$} \}
\]
hence we have
\[ 
\{d_\infty(x,0)\le1\}=\{x:\,\max\{\ep_i^i|x_i|\}\le1  \}
\]
and $x\to \max\{\ep_i^i|x_i|\}$ is convex, hence $\{d_\infty(x,0)\le1\}$ is convex.
\end{comment}
%
%
%
%
%
%
%
%
%
%
%          SECTION ON THE SHAPE OF THE HOMOGENEOUS AND PERIMETER MEASURE
%
%
%
\section{Vertically symmetric distances}\label{Sect:Hsym}
The next definition introduces those distances whose symmetries allows for having a precise geometric
constant in the definition of the spherical Hausdorff measure $\cS^{Q-1}_\G$, as discussed in the introduction.
Theorem~\ref{the:constBeta} below will prove this fact.
\begin{Def}\label{def:hsym} \rm
Let $\G$ be a stratified group of topological dimension $n$, with direct decomposition $\G=V_1\oplus W$ and
$W=V_2\oplus\cdots\oplus V_\iota$. We equip $\G$ by a scalar product that makes $V_1$ and $W$ orthogonal.
We consider a family $\cF_1\subset O(V_1)$ that acts transitively on $V_1$, where $O(V_1)$ denotes
the group of isometries of $V_1$. We set
\[ \cO_1=\{T\in GL_n(\G): T_W=\Id_W,\; T|_{V_1}\in\cF_1\}. \]
We denote by ${\bf p}_1:\G\to V_1$ the orthogonal projection onto $V_1$ and 
set \[ \B(0,1)=\{y\in\G: d(y,0)\le1\},\]
where $d$ is a homogeneous distance of $\G$. We say that $d$ is {\em $V_1$-vertically symmetric} if 
\begin{enumerate}
\item 
${\bf p}_1\ls\B(0,1)\rs=\B(0,1)\cap V_1=\{h\in V_1: |h|\le r_0\}$ for some $r_0>0$,
\item
$T\ls \B(0,1)\rs=\B(0,1)$ for all $T\in\cF_1$.
\end{enumerate}
\end{Def}
\begin{Rem}\rm
It is not difficult to observe that the distances of Example~\ref{exa:Koranyi} and Example~\ref{exa:dinfty}
are both $V_1$-vertically symmetric. The sub-Riemannian distance of the Heisenberg group is also $V_1$-vertically symmetric.
This can be also checked by the explicit formula for the profile of its sub-Riemannian unit ball.
\end{Rem}
%
%
%cmt
\begin{comment}
It seems reasonable to expect that the previous inclusion could not be turned into an equality
since $\cF_1$ is need not be a subgroup of $O(V_1)$. 
\end{comment}

%
%
%
\begin{The}\label{the:constBeta}
If a homogeneous distance $d$ is $V_1$-vertically symmetric, then $\beta(d,\cdot)$ is a constant function.
\end{The}
\begin{proof}
Let $z\in\B(0,1)$ and choose $\nu_1,\nu_2\in V_1$. Since $\cF_1$ is transitive on $V_1$ there exists
$T\in\cF_1$ such that $T(\nu_1)=\nu_2$. We set $H=\R\nu_1$ and see $\G$ as the inner semidirect 
product between $H$ and $N(\nu_1)$. In fact, we have the two canonical projections 
\[ \pi_1:\G\to H\quad \mbox{ and } \quad \pi_2:\G\to N(\nu_1) \]
 such that $y=\pi_1(y)\pi_2(y)$ for all $y\in\G$.
We set $\pi_1(z^{-1})=h$ and $\pi_2(z^{-1})=n$, therefore
\begin{equation*}
\cH^{n-1}\ls\B(z,1)\cap N(\nu_1)\rs=\cH^{n-1}\ls\B(0,1)\cap \ls h+N(\nu_1)\rs\rs\,.
\end{equation*}
By property (2) of Definition~\ref{def:hsym}, it follows that
\begin{equation*}
\cH^{n-1}\ls\B(z,1)\cap N(\nu_1)\rs=\cH^{n-1}\Big(\B(0,1)\cap \Ls Th+T\big(N(\nu_1)\big)\Rs\Big)\,.
\end{equation*}
Since $\pi_1$ is the linear projection onto $H$ with respect to the decomposition $H\oplus N(\nu_1)$
and $H\bot N(\nu_1)$, we can write $z^{-1}$ as the following sum of orthogonal vectors
\[
 w+h'+h\,,
\]
where $w\in N(\nu_1)$ and $h,h'\in V_1$. By property (1) of Definition~\ref{def:hsym}, we get
\[
 {\bf p}_1(z^{-1})=h'+h\in\B(0,1)\cap V_1=\{y\in V_1: |y|\le r_0\}.
\]
Since $h$ and $h'$ are orthogonal, we get $h\in\{y\in V_1: |y|\le r_0\}=\B(0,1)\cap V_1$, hence
\[
 T(h)\in\B(0,1)\cap V_1
\]
Since $T$ is orthogonal, $T\ls N(\nu_1)\rs=N(\nu_2)$ and we obtain
\begin{equation*}
\cH^{n-1}\ls\B(z,1)\cap N(\nu_1)\rs=\cH^{n-1}\Big(\B(0,1)\cap \Ls \big(T(h)\big) N(\nu_2)\Rs\Big)\,.
\end{equation*}
It follows that 
\begin{equation*}
\cH^{n-1}\ls\B(z,1)\cap N(\nu_1)\rs=\cH^{n-1}\big(\B\big(T(h)^{-1},1\big)\cap N(\nu_2)\big)\le\beta(d,\nu_2)\,.
\end{equation*}
The arbitrary choice of $z\in\B(0,1)$ yields $\beta(d,\nu_1)\le\beta(d,\nu_2)$. Exchanging the role of $\nu_1$ for 
that of $\nu_2$, we conclude the proof.
\end{proof}

\vskip3mm
\noindent
{\bf Acknowledgements.} The author thanks Luigi Ambrosio for fruitful conversations.

\end{document}